\documentclass[a4paper,12pt,onecolumn]{article}

\usepackage[vmargin=2cm,hmargin=2cm,headheight=14.5pt,top=2cm,headsep=.5cm]{geometry}
\usepackage{bm}
\usepackage{empheq}
\usepackage{stackrel}
\usepackage{cases}
\usepackage{mathtools}
\usepackage{amsthm,amsmath,amscd}
\usepackage{makeidx}
\usepackage{perpage}
\usepackage{framed}
\usepackage[symbol]{footmisc}

\MakePerPage[2]{footnote}
\usepackage{graphicx}
\DeclareMathSizes{12}{12}{8}{6}

\usepackage{cite}
\usepackage{url}
\usepackage[charter]{mathdesign}
\usepackage{accents}
\usepackage{hyperref}

\newtheoremstyle{ptheorem}{1em}{0em}{\itshape}{}{\bfseries}{.}{.5em}{}

\theoremstyle{ptheorem}

\newtheorem{thm}{Theorem}[section]

\newtheorem{lem}[thm]{Lemma}

\theoremstyle{definition}

\newtheorem{dfn}[thm]{Definition}
\newtheorem{rem}[thm]{Remark}

\theoremstyle{remark}

\newtheorem{exa}[thm]{Example}

\begin{document}
\title{Fixed points of Hammerstein-type equations on general cones\footnote{Partially supported by  Ministerio de Econom\'ia y Competitividad (Spain) project MTM2013-43014-P and Xunta de Galicia (Spain), project EM2014/032.}}
\author{
 Rub\'en Figueroa\\
	\normalsize e-mail: ruben.figueroa@usc.es\\
	F. Adri\'an F. Tojo\footnote{Supported by  FPU scholarship, Ministerio de Educaci\'on, Cultura y Deporte, Spain.} \\
	\normalsize e-mail: fernandoadrian.fernandez@usc.es\\
	\normalsize \emph{Departamento de An\'alise Ma\-te\-m\'a\-ti\-ca, Facultade de Matem\'aticas,} \\ \normalsize\emph{Universidade de Santiago de Com\-pos\-te\-la, Spain.}\\ 
}
\date{}
\maketitle

\begin{abstract}
	{We obtain new results on the existence and multiplicity of fixed points of Hammerstein e\-qua\-tions in very general cones. In order to achieve this, we combine a new formulation of cones in terms of continuous functionals with fixed point index theory. Many examples and an application to boundary value problems are also included.}
\end{abstract}


\noindent {\bf Classification}: 37C25; 47H30; 34B15.

\bigskip

\section{Introduction}

In the last years, a vast amount of literature devoted to fixed point index theory in cones has been written. Ever since the publication of the well-known Krasnosel'ski\u{\i}'s Fixed Point Theorem~\cite{Kras60}, some authors have attempted to obtain new results in order to generalize and apply it to a large class of problems \cite{Alg,Kwo,Pet}. Probably, one of the most useful applications of Krasnosel'ski\u{\i}-type theorems is the localization of solutions of differential equations satisfying certain boundary conditions \cite{Berzig2014,CAC,Dha}. A classical approach in this direction consists in rewriting the original di\-ffe\-ren\-tial problem in terms of an operator defined in a normed space. The next step is to use some fixed-point technique to ensure that the operator has a fixed point that will correspond to a solution of the boundary value problem. \\

In the light of this background, we develop a unified framework that allows us to look for solutions in a large class of boundary value problems. As it is well-known, most of these problems can be rewritten in terms of a {\it Hammerstein}-type equation, so our goal will be to obtain new results on the existence and localization of fixed points for this equation. On this ground, we develop a new general formulation to obtain abstract Krasnosel'ski\u{\i}-type results in general cones. \\

The paper is organized as follows: in Section 2 we deal with abstract cones in normed spaces and show how these sets can be characterized in terms of continuous functionals; also, we include many examples of the application of this new perspective to some of the cones which are most often used in the literature. In Section 3 we obtain the main results of this work, which are about the existence and localization of solutions of Hammerstein-type equations in cones. Finally, in Section 4 we illustrate the theory providing an example to which we apply our results.

\section{Characterization of cones in terms of functionals}

We begin by recalling some concepts about cones in normed spaces.

\begin{dfn} Let $(N,\|\cdot\|)$ be a real normed space. A \emph{cone} in $N$ is a closed set such that
\begin{enumerate}
\item $u+v\in K \ \mbox{ for all $u,v \in K$}$;
\item $\lambda\, u\in K\ \mbox{ for all $u \in K$, $\lambda \in [0,+\infty)$}$;
\item $K\cap(-K)=\{0\}$.
\end{enumerate}
\end{dfn}
In the sequel, $(N,\|\cdot\|)$ will denote a real normed space and $$N^*:=\{\alpha :N\to\mathbb{R}\ :\ \alpha \text{ continuous}\}$$ will be the set of continuous functionals defined on $N$. Moreover, we will consider $\mathcal{A}\subset N^*$ to be the set of those $\alpha \in N^*$ which satisfy the following three conditions:
\begin{equation}
 \alpha (u+v)\ge\alpha (u)+\alpha (v) \ \mbox{ for all $u,v \in N$}; \label{a1}\end{equation}
\begin{equation} \alpha (\lambda u)\ge\lambda\alpha (u) \ \mbox{ for all $u\in N ,\ \lambda\in[0,+\infty)$};\label{a2}\end{equation}
\begin{equation} [\alpha (u)\ge 0,\alpha (-u)\ge 0]\Rightarrow u=0.\label{a3}
\end{equation}

Notice that, in general, we cannot ensure that $\mathcal{A}$ is a vector subspace of $N^*$. However, it follows from (\ref{a1})--(\ref{a3}) that if $\alpha, \beta \in \mathcal{A}$ and $\lambda \in [0,+\infty)$ then $\min\{\alpha,\beta\} \in \mathcal{A}$ and $\lambda \alpha \in \mathcal{A}$. \\

Condition (\ref{a3}) could be quite difficult to check in practice. Nevertheless, notice that a suf\-fi\-cient condition to guarantee that (\ref{a3}) is satisfied is the following:

\begin{equation}\label{a4}\alpha (u)+\alpha (-u)\le0 \ \mbox{ for all $u\in N$ and } \alpha (u)=\alpha (-u)=0 \ \mbox{ implies $u=0$.}\end{equation}


The following Lemma will be useful in subsequent applications. In the sequel, given $\alpha  \in N^*$ we will denote $\widetilde{\alpha }(u)=\alpha (-u)$.

\begin{lem}\label{lemsc} Let $\{\alpha _j\}_{j\in J}\subset N^*$ such that $\displaystyle{\sum_{j\in J}}\alpha _j\in N^*,$ $u\in N$. Assume that:
\begin{equation}\label{a5.1}\alpha _j+\widetilde{\alpha_j}\le0 \ \mbox{ for all } j\in J,\end{equation}
and
\begin{equation}\label{a5.2}\bigcap_{j \in J}(\alpha _j+\widetilde{\alpha _j})^{-1}(\{0\})=\{0\}.\end{equation}
Then $\displaystyle{\sum_{j\in J}\alpha _j}$ satisfies condition \eqref{a4}.
\end{lem}
\begin{proof} By condition \eqref{a5.1}, we have that $\displaystyle{\sum_{j\in J}\alpha _j+\sum_{j\in J}\widetilde{\alpha _j}\le0}$. Now, let $u\in N$ be an element satisfying $\displaystyle{\sum_{j\in J}\alpha _j(u)=\sum_{j\in J}\alpha _j(-u)=0.}$ This implies that $\displaystyle{\sum_{j\in J}[\alpha _j(u)+\widetilde{\alpha _j}(u)]=0}$ and thus, because of \eqref{a5.1}, $\alpha _j(u)+\widetilde{\alpha _j}(u)=0$ for every $j\in J$. Hence we have $u\in(\alpha _j+\widetilde{\alpha _j})^{-1}(\{0\})$ for every $j\in J$. Therefore,\[u\in\stackbin[j\in J]{}{\bigcap}(\alpha _j+\widetilde{\alpha _j})^{-1}(\{0\}),\] and, by virtue of \eqref{a5.2}, condition \eqref{a4} is satisfied.
\end{proof}

Now we introduce the main result of this section, which characterizes all cones in $N$ in terms of suitable functionals. For this purpose, we denote by $\mathcal{K}$ be the set of all cones in $N$ and, given $\alpha  \in \mathcal{A}$, we define $K_\alpha :=\{u\in N\ :\ \alpha (u)\ge0\}$.
\begin{rem} With the notation introduced above, it is clear that $$K_\alpha \cap K_\beta=K_{\min\{\alpha ,\beta\}}$$ for $\alpha ,\beta\in\mathcal{A}$. In the same way, $\stackbin[\alpha \in\mathcal{A}']{}{\bigcap}K_\alpha =K_{\inf\limits_{\alpha \in\mathcal{A}'}\alpha }$ for every $\mathcal{A}'\subset\mathcal{A}$. In \cite{gijwems} we can see a cone constructed in this way.
\end{rem}
\begin{thm}$\mathcal{K}=\{K_\alpha  \, : \, \alpha \in\mathcal{A}\}$.
\end{thm}
\begin{proof}
To see that $\mathcal{K}\supset\{K_\alpha ,\ \alpha \in\mathcal{A}\}$ we only have to notice that for every $\alpha \in\mathcal{A}$, $K_\alpha $ is a cone by properties $\eqref{a1}-\eqref{a3}$.\par
Now we show that $\mathcal{K}\subset\{K_\alpha ,\ \alpha \in\mathcal{A}\}$. Let $K\in\mathcal{K}$ and define \[\alpha (u):=-\inf_{w\in K}\|u-w\|.\] Thus defined, $\alpha $ is a continuous functional ($\alpha (u)$ is actually minus the distance from $u$ to the cone $K$). Clearly, $K_\alpha =K$. Furthermore, for $u,v\in N$ and $\lambda\in[0,+\infty)$, we have that
\begin{align*}\alpha (u+v) & =-\inf_{w\in K}\|u+v-w\|=-\inf_{w\in K}\|u+v-2w\| \\ & \ge-\inf_{w\in K}(\|u-w\|+\|v-w\|)=\alpha (u)+\alpha (v),\\
\alpha (\lambda\,u) & =-\inf_{w\in K}\|\lambda\,u-w\|=-\inf_{w\in K}\|\lambda\,u-\lambda\,w\|=-\inf_{w\in K}\lambda\|u-w\|=\lambda\alpha (u).\end{align*}
Finally, if $u\in K\backslash\{0\}$, $-u\not\in K$ and so $\alpha (u)<0$. Therefore $\alpha \in\mathcal{A}$.
\end{proof}
\begin{rem}In the previous result we have proved something even stronger: we can take $\alpha $ to satisfy $\alpha (\lambda u)=\lambda\alpha (u)$ and $\alpha (u)=0$ for every $u\in K$, $\lambda\in[0,+\infty)$.
\end{rem}
This last result shows that any cone on a normed space is given by a functional satisfying pro\-per\-ties $\eqref{a1}-\eqref{a3}$. Now, we may wonder under which circumstances two different functionals define the same cone. In order to elucidate this, given a cone $K$ in $N$, define the functional $\varphi_K(u):=d(u,\partial K)-2d(u,K)$ where $\partial K$ is the boundary of $K$ and $d(u,X):=\inf_{w\in X}\|u-w\|$ is the distance from $u$ to the set $X\subset N$. The way it is defined, $\varphi_K$ is clearly continuous. Actually, we have that
\[\varphi_K(u)=\begin{dcases}-d(u,K)<0, &  u\in N\backslash K, \\ d(u,\partial K)>0, & u\in\operatorname{Int}(K),\\ 0, & u\in\partial K.
\end{dcases}\]
With this, it is easy to prove the following Lemma.
\begin{lem}Let $\alpha ,\beta\in\mathcal{A}$. Then $K_\alpha =K_\beta$ if and only if $\beta=\xi\,\varphi_{K_\alpha }$ for some $\xi:N\to[0,+\infty)$ such that $\xi(u)>0$ for $u\in N\backslash K_\alpha $.
\end{lem}
\begin{proof} Assume $K_\alpha =K_\beta$ and define $\xi=\beta/\varphi_{K_\alpha }$ in $N\backslash \partial K_\alpha $ and $\xi=0$ in $\partial K_\alpha $. $\xi$ is well defined since $\varphi_{K_\alpha }\ne0$ in $N\backslash \partial K_\alpha $. Clearly, $\beta=\xi\varphi_{K_\alpha }$ in $N\backslash \partial K_\alpha $. Also, since $\beta$ is continuous, $\beta\ge 0$ in $K_\alpha $ and $\beta<0$ in $N\backslash K_\alpha $, we have that $\beta=0$ in $\partial K_\alpha $, so $\beta=\xi\varphi_{K_\alpha }$ in $N$.\par
Assume now $\beta=\xi\varphi_{K_\alpha }$. Then $\xi\varphi_{K_\alpha }\ge0$ in $K_\alpha $ and so $K_\alpha \subset K_\beta$. On the other hand, if $u\in N\backslash K_\alpha $ then $\xi(u)>0$ and $\varphi_{K_\alpha }<0$, so $\xi\varphi_{K_\alpha }<0$ and $u\in N\backslash K_\beta$. Hence, $K_\alpha =K_\beta$.
\end{proof}

Now we show some examples of functionals satisfying \eqref{a1} -- \eqref{a3}. As we will see, these functionals will be related to some cones which frequently appear in the literature. \\

In the following, consider the interval $I=[0,1]$ and the Banach space of continuous functions with the maximum norm $(\mathcal{C}(I),\, \|\cdot\|)$.

\begin{exa} Let $\|\cdot\|_*$ be a continuous norm (possibly different from the maximum norm $\|\cdot\|$) in $\mathcal{C}(I)$, $K$ be a cone in $\mathcal{C}(I)$ and $\sigma\in\mathcal{C}(I)$ a positive function. The functional
\[\alpha (u):=-\inf_{v\in K}\|\sigma\,u-v\|_*\]
satisfies properties \eqref{a1} and \eqref{a2}. If $\alpha (u)\ge0$ and $\alpha (-u)\ge0$, then $\alpha (u)=\alpha (-u)=0$. Then, being $\alpha $ continuous and $K$ closed,   there exist $v,w\in K$ such that $\|\sigma\,u-v\|_*=\|-\sigma\,u-w\|_*=0$, so $\sigma\,u=v=-w$. Since $K$ is a cone, $K\cap(-K)=\{0\}$, and hence $\sigma\,u=v=-w=0$. $\sigma$ is positive, which implies $u=0$. \par
From the above discussion we deduce that $K_\alpha =\{u/\sigma\ :\ u\in K\}$.\
For the particular choices $\sigma=1$, $K=\{0\}$ and $\|\cdot\|_*=c\|\cdot\|$, $c\in(0,+\infty)$, we have $\alpha (u)=-c\|u\|$.\par
Also, we can choose $\|\cdot\|_*=\|\cdot\|_p$, the $p$ norm of $\operatorname{L^p}(I)$, to have $\alpha (u)=-\|\cdot\|_p$.
\end{exa}

\begin{exa} For $u\in\mathcal{C}(I)$, let $\max u$ ($\min u$) be the maximum (minimum) of $u$ on its domain.
It is clear that the functionals $\min u$, $-\max u$, $-\|u\|$, satisfy conditions \eqref{a1}, \eqref{a2} and \eqref{a4}. In fact, for a function $\sigma\ge0$ we can generalize this to the functionals $\min (\sigma\, u)$, $-\max (\sigma\, u)$, $-\|\sigma\,u\|$, which satisfy properties \eqref{a1} and \eqref{a2}. If we define $\chi_{[a,b]}$ to be the characteristic function of the interval $[a,b]\subset I$ and take $c\in(0,+\infty)$, we can combine the above functionals using Lemma \ref{lemsc} to derive the  functional
\[\alpha (u)=\min(\chi_{[a,b]}\, u)-c\|u\|,\]
which satisfies conditions  \eqref{a1}, \eqref{a2} and \eqref{a4}.
Observe that, for $u\in K_\alpha $, \[\alpha (u)=\min_{t\in[a,b]}u(t)-c\|u\|,\]  and in this form it is used in \cite{jw-tmna,Franco,gippmt,Wan,Goodrich4,Goodrich5,Goodrich3,IPZ,jwkleig,Don,Wang1,jwmz-na,Li-Cong,sun2,jw-gi-jlmsII,gijwems,ma-non,gi-pp1,gi-pp,Sun3,WIF,jw-gi-jlms,webb-semi,IP-disp,Sun,gijwjiea,CITZ,WebbUni}.
\par
We can derive, in the same way, the more general functional \[\alpha (u)=\min(\chi_{[a,b]}\,\sigma\, u)-\|u\|,\] where $\sigma \in\mathcal{C}(I)$, $\sigma >0$, which appears in the cones of the form
\[K=\{u\in\mathcal{C}(I)\ :\ u(t)\ge \sigma(t)\|u\|,\ t\in I\},\]
used in \cite{jw-tmna,Erbe,Anu,Lan2}.
\end{exa}
\begin{exa} Let $t_0\in I$, $a,b\in\mathcal{C}(I)$, $a,b\ge0$, $a+b>0$. The functional \[\alpha (u)=\min(a\,u)-\max(b\,u)-|u(t_0)|,\] satisfies  properties \eqref{a1} and \eqref{a2}. Also,
\[\alpha (u)+\widetilde{\alpha }(u)=\min(a\,u)-\max(a\,u)+\min(b\,u)-\max(b\,u)-2|u(t_0)|\le 0,\]
and $\alpha (u)+\widetilde{\alpha }(u)=0$ only for $u=0$, so condition \eqref{a4} is also satisfied.
\end{exa}

\begin{exa} Consider \[\nparallel u\nparallel:=\max\left\{\min u,-\max u\right\}.\]
For every $u,v\in\mathcal{C}(I)$ and $\lambda\in\mathbb{R}$, the functional $\nparallel \cdot\nparallel$ satisfies the following conditions:
\begin{itemize}
\item $\nparallel u+v\nparallel\ge \nparallel u\nparallel+\nparallel v\nparallel$,
\item $\nparallel \lambda u\nparallel=|\lambda|\nparallel  u\nparallel$,
\item $\|u\|-\nparallel u\nparallel=\max u-\min u\ge0$,
\item $\nparallel u\nparallel-\min u=\|u\|-\max u\ge0$,
\item $\max\{\max u,\nparallel u\nparallel\}=|\max u|$.
\end{itemize}
Consider then the functional \[\alpha (u)=\nparallel a\,u\nparallel-\|b\,u\|,\]
where $a,b\in\mathcal{C}(I)$ are such that $|a|\le|b|$ and $|a|+|b|>0$. $\alpha $ satisfies  conditions \eqref{a1} and \eqref{a2}. Then,
\[\alpha (u)=\|a\,u\|-\|b\,u\|+\min(a\,u)-\max(a\,u)\le 0,\]
In fact, if $\alpha (u)=0$, then $\|a\,u\|-\|b\,u\|=0$ and $\min(a\,u)-\max(a\,u)=0$. Therefore $a\,u=b\,u=0$. Since $|a|+|b|>0$ we obtain $u=0$ and so condition \eqref{a4} is satisfied.
\end{exa}
\begin{exa} Let $S\subset \mathcal{C}(I)$ be a bounded set such that for every $t\in I$ there exists $\sigma\in S$ satisfying $\sigma(t)\ne 0$ in an open neighborhood of $t$. Also, assume $\stackbin[\sigma\in S]{}{\bigcup}\sigma(I)$ has at least two elements. Define
\[\alpha (u)=\inf_{\sigma\in S}\min(\sigma\,u).\]
Thus defined, $\alpha $ satisfies  conditions \eqref{a1} and \eqref{a2} and also
\[\alpha (u)+\widetilde{\alpha }(u)=\inf_{\sigma\in S}\min(\sigma\,u)-\sup_{\sigma\in S}\max(\sigma\,u)\le \inf_{\sigma\in S}\left(\min(\sigma\,u)-\max(\sigma\,u)\right)\le0.\]
Now, assume $\alpha (u)+\widetilde{\alpha }(u)=0$. This implies $\min(\sigma\,u)=\max(\sigma\,u)$ for every $\sigma\in S$, so $\sigma u$ is constant for every $\sigma\in S$. Furthermore, for every $t$ there exists $\sigma\in S$ such that $\sigma(t)\ne 0$ in an open neighborhood of $t$, so $u$ is constant in an open neighborhood of $t$. Consider the set $A=u^{-1}(\{u(0)\})$. Being the inverse image by a continuous function of a closed set, $A$ is closed in $I$. On the other hand, $A$ is open in $I$, since for every $t\in A$ there is a neighborhood $U$ of $t$ such that $u$ is constant in $U$. Then $u(t)=u(t_0)$ for all $t \in U\subset A$. As $A$ is both closed and open in $I$, $A=I$, so $u$ is constant in all of $I$. Hence, \[\alpha (u)+\widetilde{\alpha }(u)=u(0)\left[\inf_{\sigma\in S}\min\sigma-\sup_{\sigma\in S}\max\sigma\right].\]\par
Now, there exist $\sigma_1,\sigma_2\in S$ and $t_1,t_2\in I$ such that $\sigma_1(t_1)>\sigma_2(t_2)$. Hence, we obtain $\inf_{\sigma\in S}\min\sigma-\sup_{\sigma\in S}\max\sigma<0$ and therefore $u=0$.
\end{exa}
\begin{exa}Consider now a function $h:I\times\mathbb{R}\to\mathbb{R}$ such that for $u \in \mathcal{C}(I)$ the composition $t \in I \longmapsto h(t,u(t))$ is integrable and which moreover satisfies that $h(t,x+y)\ge h(t,x)+h(t,y)$ and $h(t,\lambda\, x)\ge\lambda h(t,x)$ for $x\in\mathbb{R}$, $t\in I$ and $\lambda\ge0$ (we could consider, for instance, the function $h(t,x)=e^t\chi_{[0,+\infty)}(t)x$ where $\chi$ is the characteristic function). Consider also a positive measure given by a function of bounded variation $A$ and the functional given by the Stieltjes integral
\[\alpha (u)=\int_0^1h(t,u(t))\operatorname{d} A(t).\]
\par
Cones defined by functionals involving integrals can be found in a number of works, for instance \cite{ma-non,WZFP,ZA2011}, and functionals given by a measure of bounded variation in \cite{gi-pp-ft,gippmt,Goodrich8,Jan4,jw-gi-jlms,jwmz-na,webb-semi,WIF}.
\end{exa}
\begin{exa}\label{conex}
The set of continuous concave functions is given by $$C=\{u\in\mathcal{C}(I)\ :\ \alpha (u)\ge0\},$$ where
\[\alpha (u):=\inf_{t,s\in I}\left[u\left(\frac{t+s}{2}\right)-\frac{u(t)+u(s)}{2}\right]\footnote{Actually, this functional provides the set of mid-point concave continuous functions. A theorem of Sierpi\'nski (see \cite{Donoghue}) shows that in the case of measurable functions, mid-point concave and concave functions coincide.}.\]
The functional $\alpha $ satisfies conditions \eqref{a1} and \eqref{a2}. If $u\in\mathcal{C}(I)$ and $\alpha (u),\alpha (-u)=0$ then we have what is called Jensen's functional equation:
\[u\left(\frac{t+s}{2}\right)=\frac{u(t)+u(s)}{2} \mbox{ for all } t,s\in I.\]
To solve it, just define $I_n:=[0,2^{-n}]$, $n=0,1,\dots$, and observe that, for $t\in I_{n}$ and $s=2^{-n}-t$ we have that
\[u\left(2^{-n-1}\right)=\frac{u(t)+u(2^{-n}-t)}{2};\ t,s\in I.\]
That is, $u$ is symmetric with respect to $2^{-n-1}$ in the interval $I_n$, which means that $u(I_{n+1})=u(I_n)$ for every $n=0,1,\dots$ or, equivalently, $u(I)=\stackrel[k=0]{\infty}{\bigcap} u(I_n)$ for every $n=0,1,\dots$ If $y\in\stackrel[k=0]{\infty}{\bigcap} u(I_n)$ for every $n=0,1,\dots,$ there exists $x_n\in I_n$ such that $u(x_n)=y$. Since $0\le x_n\le 2^{-n}$, we have that $x_n\to 0$ and, since $u$ is continuous, $u(x_n)\to u(0)$. Therefore, $y=u(0)$ and so $u$ is a constant. Reciprocally, every constant satisfies Jensen's equation and, in conclusion, $C$ is not a cone.\par
Now, if we consider $\eta\in I$ and define the closed vector subspace $$N_\eta:=\{u\in\mathcal{C}(I)\ :\ u(\eta)=0\},$$ we have that $\mathcal{C}(I)=N_\eta\oplus\mathbb{R}$ and in this case $C\bigcap N_\eta$ is a cone.\par
Cones in which concave functions are involved appear, for instance, in \cite{AP2001,JLY2012}.

\end{exa}
\section{Fixed point results for Hammerstein equations}

In this section we obtain some results regarding the existence of solutions of integral equations of Hammerstein-type in abstract cones. To do this, we will work in cones characterized by functionals satisfying (\ref{a1})--(\ref{a3}). Consider again the interval $I:=[0,1]$ and the Banach space of continuous functions with the maximum norm $(\mathcal{C}(I),\, \|\cdot\|)$. Given a functional $\alpha :\mathcal{C}(I)\to\mathbb{R}$, $\alpha \in\mathcal{A}$, we look for fixed points in $K_{\alpha }$ of an operator $T:\mathcal{C}(I)\to\mathcal{C}(I)$ given by\begin{equation}\label{eqhamm}
Tu(t):=\int_{0}^{1} k(t,s)g(s)f(s,u(s))\operatorname{d} s.
\end{equation}

An equation of the form (\ref{eqhamm}) is usually known as a \emph{Hammerstein}-type equation, and there are many papers in the literature which deal with this type of equations, see for instance \cite{CCI,Cai,Chi}. Typically, as we said in Section 1, these equations appear when looking for solutions of certain type of boundary value problems. In this context, the kernel $k$ uses to be the Green's function of a related problem and $g$ and $f$ are, respectively, the linear and the nonlinear part of the differential equation in that problem. In this context, our work provides a new point of view from which all these problems can be considered in a unified framework. \\

The way we look for solutions of equation (\ref{eqhamm}) is the well-known technique in fixed point index theory. For the sake of completeness, we recall now a classical result for continuous compact maps (cf. \cite{amann} or \cite{guolak}).

Let $K$ be a cone in $\mathcal{C}(I)$. If $\Omega$ is a bounded open subset of $K$ (in the relative
topology) we now denote by $\overline{\Omega}$ and $\partial \Omega$ respectively
its closure and its boundary. Moreover, we will denote $\Omega_K=\Omega \cap K$, which is an open subset of
$K$.

\begin{lem}\label{lemind}
Let $\Omega$ be an open bounded set with $0\in \Omega_{K}$ and $\overline{\Omega_{K}}\ne K$. Assume that $F:\overline{\Omega_{K}}\to K$ is
a continuous compact map such that $x\neq Fx$ for all $x\in \partial \Omega_{K}$. Then the fixed point index\index{fixed point index} $i_{K}(F, \Omega_{K})$ has the following properties.
\begin{itemize}
\item[(1)] If there exists $e\in K\backslash \{0\}$ such that $x\neq Fx+\lambda e$ for all $x\in \partial \Omega_K$ and all $\lambda
>0$, then $i_{K}(F, \Omega_{K})=0$.
\item[(2)] If  $\mu x \neq Fx$ for all $x\in \partial \Omega_K$ and for every $\mu \geq 1$, then $i_{K}(F, \Omega_{K})=1$.
\item[(3)] If $i_K(F,\Omega_K)\ne0$, then $F$ has a fixed point in $\Omega_K$.
\item[(4)] Let $\Omega^{1}$ be open in $X$ with $\overline{\Omega^{1}}\subset \Omega_K$. If $i_{K}(F, \Omega_{K})=1$ and
$i_{K}(F, \Omega_{K}^{1})=0$, then $F$ has a fixed point in $\Omega_{K}\backslash \overline{\Omega_{K}^{1}}$. The same result holds if
$i_{K}(F, \Omega_{K})=0$ and $i_{K}(F, \Omega_{K}^{1})=1$.
\end{itemize}
\end{lem}

Now we state the main results of this paper. In order to do so, we consider the following list of assumptions for equation \eqref{eqhamm} and the cone $K_\alpha $ with $\alpha \in\mathcal{A}$:

\begin{enumerate}
\item [$(C_{1})$] The kernel $k$ is measurable and the function $k(\cdot,s)$ is uniformly continuous with respect to $s$, that is, for every $\varepsilon >0$ there exists $\delta >0$ such that $|t_1-t_2| < \delta$ implies $|k(t_1,s)-k(t_2,s)| < \varepsilon$ for all $s \in I$;
\item [$(C_{2})$]
 $\psi_\alpha (s):=\alpha (k(\cdot,s))\ge0$ for a.\,a. (almost all) $s\in I$;
\item [ $(C_{3})$] the functions $g$, $k(t,\cdot)g$ and $\psi_\alpha \, g$ are integrable and $g(t) \geq 0$ for a.\,a. $t \in I$;
\item  [ $(C_{4})$] the nonlinearity $f:I\times \mathbb{R}
 \to [0,+\infty)$ satisfies  $\operatorname{L^\infty}$-Carath\'{e}odory
conditions, that is, $f(\cdot,u)$ is measurable for each fixed
$u$ and $f(t,\cdot)$ is continuous for a.\,a. $t\in I$, and, for each $r>0$, there exists $\phi_{r} \in
\operatorname{L^\infty}(I)$ such that $f(t,u)\le \phi_{r}(t)$ for all $u\in
[-r,r]$ and a.\,a. $t\in I$;
\item  [$(C_{5})$] \[\alpha (Tu)\ge\int_0^1\psi_\alpha (s)g(s)f(s,u(s))\operatorname{d} s \ \mbox{ for all $u\in\ K_\alpha $;}\]
\item  [$(C_{6})$]
there exist two continuous functionals $\beta,\gamma:\mathcal{C}(I)\to\mathbb{R}$ satisfying that, for $u,v\in K_\alpha $ and $\lambda\in[0,+\infty)$,

\[\beta(u+v)\le\beta(u)+\beta(v),\ \beta(\lambda\,u)=\lambda\beta(u),\ \beta(Tu)\le\int_0^1\psi_\beta(s)g(s)f(s,u(s))\operatorname{d} s,\]
\[\gamma(u+v)\ge\gamma(u)+\gamma(v),\ \gamma(\lambda\,u)\ge\lambda\gamma(u),\ \gamma(Tu)\ge\int_0^1\psi_\gamma(s) g(s)f(s,u(s))\operatorname{d} s,\]
\[\psi_\beta,\psi_\gamma\in L^1(I) \ \mbox{ and }\int_0^1\psi_\beta(s)g(s)\operatorname{d} s,\int_0^1\psi_\gamma(s)g(s)\operatorname{d} s>0;\]
\item  [$(C_{7})$] there exists $e\in K_\alpha \backslash\{0\}$ such that $\gamma(e)\ge0$;
\item  [$(C_{8})$] for every $\rho>0$ there exist $b(\rho),c(\rho)>0$ such that $\beta(u)\le b(\rho)$ for every $u\in K_\alpha $ satisfying $\gamma(u)\le \rho$ and $\gamma(u)\le c(\rho)$ for every $u\in K_\alpha $ satisfying $\beta(u)\le \rho$.
\end{enumerate}

\begin{rem}
Notice that if the kernel $k(t,s)$ is a.\,e. differentiable with respect to $t$ and $\partial k/\partial t$ is uniformly bounded with respect to $s$ then condition $(C_1)$ is satisfied. This formulation is useful in applications when $k$ corresponds to a Green's function.
\end{rem}

\begin{thm}\label{thmk}
Assume hypotheses $(C_{1})$--$(C_{5})$. Then $T$ is continuous, compact and maps $K_\alpha $ to $K_\alpha $.
\end{thm}
\begin{proof}
Continuity and compactness are derived from standard arguments involving Lebesgue's Do\-mi\-na\-ted Convergence Theorem, but we include it for completeness.

\emph{Continuity:} Let $\{u_n\}_{n \in \mathbb{N}}$ be a sequence which converges to $u$ in $\mathcal{C}(I)$. In particular, $\{u_n\}_{n \in \mathbb{N}}$ is bounded, that is, there exists $r >0$ such that $||u_n|| \le r$ for all $n \in \mathbb{N}.$ Moreover, we have by virtue of  $(C_4)$ that $f(s,u_n(s)) \to f(s,u(s))$ for a.\,e. $s \in I$. Then, conditions $(C_3)-(C_4)$ imply now that
$$
|Tu_n(t)| \le \|\phi_r\| \int_0^1 |k(t,s) \, g(s)| \, \operatorname{d} s \ \mbox{ for all $t \in I$},$$
and we obtain, by application of Lebesgue's Dominated Convergence Theorem that $Tu_n \to Tu$, in $\mathcal{C}(I)$. Hence, operator $T$ is continuous.

\emph{Compactness:} Let $B \subset K_{\alpha }$ a bounded set, that is, $||u|| \le R$ for all $u \in B$ and some $R>0$. Then similar arguments as above show that
$$
|Tu(t)| \le \|\phi_R\| \int_0^1 |k(t,s) \, g(s)| \, \operatorname{d} s \ \mbox{ for all $t \in I$ and all $u \in B$}.$$
Therefore, the continuity of $k(\cdot,s)$ implies that the set $T(B)$ is totally bounded. On the other hand, given $t,s \in I$, we have
$$
|Tu(t)-Tu(s)| \le  \int_0^1 |k(t,r)-k(s,r)| \, g(r) \, \phi_R(r) \, \operatorname{d} r,$$
Hence, by virtue of $(C_1)$, $(C_3)$ and $(C_4)$,  $T(B)$ is equicontinuous. In conclusion, we derive, by application of Ascoli--Arzela's Theorem, that $T(B)$ is relatively compact in $\mathcal{C}(I)$ and derive that $T$ is a compact operator.

Finally, we obtain from conditions $(C_2)$ and $(C_5)$ that
\[\alpha (Tu)\ge\int_0^1\psi_\alpha (s)g(s)f(s,u(s))\operatorname{d} s\ge0 \mbox{ for all $u \in K_{\alpha }$.}\]
Thus, $Tu\in K_\alpha $.
\end{proof}

In the sequel, we give a condition that ensures that, for a suitable $\rho>0$, the index is $1$ or $0$ in certain open subsets of $K_{\alpha }$. In order to see this, we define the sets \begin{align*}K_\alpha ^{\beta,\,\rho}:= & \beta^{-1}([0,\rho))\cap K_\alpha =\{u\in\mathcal{C}(I)\ :\ \alpha (u)\ge0,\ 0\le\beta(u)<\rho\},\\
K_\alpha ^{\gamma,\,\rho}:= & \gamma^{-1}([0,\rho))\cap K_\alpha =\{u\in\mathcal{C}(I)\ :\ \alpha (u)\ge0,\ 0\le\gamma(u)<\rho\}.\end{align*}
We can define now two functions $b,c:\mathbb{R}^+\to\mathbb{R}^+$ in the conditions of $(C_8)$ in the following way:
\[b(\rho):=\sup\{\beta(u)\ :\ u\in K_\alpha ,\ \gamma(u)<\rho\},\quad c(\rho):=\sup\{\gamma(u)\ :\ u\in K_\alpha ,\ \beta(u)<\rho\}.\]
With these definitions, $K_\alpha ^{\beta,\,\rho}\subset K_\alpha ^{\gamma,\,c(\rho)}$ and $ K_\alpha ^{\gamma,\,\rho}\subset K_\alpha ^{\beta,\,b(\rho)}$.
\begin{lem}
\label{ind1b} Assume that
\begin{enumerate}
\item[$(\mathrm{I}_{\protect\rho }^{1})$] \label{EqB} there exists $\rho> 0$ such that
\begin{displaymath}
 f^{\rho}\cdot\int_0^1\psi_\beta(s)g(s)\operatorname{d} s
 <1,
\end{displaymath}
where
\begin{displaymath}
  f^\rho=\sup \left\{\frac{f(t,u(t))}{\rho }:\;t\in
I,\ u\in K_\alpha ,\ \beta(u)=\rho\right\}.\end{displaymath} \end{enumerate}{}
Then the fixed point index $i_{K}(T,K_\alpha ^{\beta,\rho})$ is equal to $1$.
\end{lem}
\begin{proof}
We show that $\mu u \neq Tu$ for every $u \in \partial K_\alpha ^{\beta,\rho}=\beta^{-1}(\rho)\cap K_\alpha $
and for every $\mu \geq 1$. 
In fact, if this does not happen there exist $\mu \geq 1$ and $u\in
\partial K_\alpha ^{\beta,\rho}$ such that $\mu u=Tu$,
that is
\begin{displaymath}\mu u(t)= \int_{0}^{1}
k(t,s)g(s)f(s,u(s))\operatorname{d} s.\end{displaymath}
Taking $\beta$ on both sides,
\begin{align*}
\mu\beta(u)=\mu\rho\leq
\int_0^1\psi_\beta(s)g(s)f(s,u(s))\operatorname{d} s
 \leq\rho f^{\rho}\cdot\int_0^1\psi_\beta(s)g(s)\operatorname{d} s
 <\rho.
\end{align*}

This
contradicts the fact that $\mu \geq 1$ and proves the result.
\end{proof}

\begin{lem}
\label{idx0b1} Assume that

\begin{enumerate}
\item[$(\mathrm{I}_{\protect\rho }^{0})$] there exist $\rho >0$ such that
such that
\begin{displaymath}
f_\rho\cdot\int_0^1\psi_\gamma(s)g(s)\operatorname{d} s
 >1,
\end{displaymath}
where
\begin{displaymath}
f_\rho =\inf \left\{\frac{f(t,u(t))}{\rho }:\;t\in I,\ u\in K_\alpha ,\ \gamma(u)=\rho\right\}.\end{displaymath}
\end{enumerate}

Then $i_{K}(T,K_\alpha ^{\gamma,\rho})=0$.
\end{lem}
\begin{proof} Take $e$ as in $(C_7)$. Now we show that
$u\ne Tu+\lambda e$ for every $u\in \partial K_\alpha ^{\gamma,\rho}=\gamma^{-1}(\rho)\cap K_\alpha $ and $\lambda \geq 0.$ Assume otherwise that there exist $u\in \partial K_\alpha ^{\gamma,\rho}$ and $\lambda\geq 0$
such that $u=Tu+\lambda e$. Then we have
\begin{displaymath}
u(t)=\int_0^1
k(t,s)g(s)f(s,u(s))\operatorname{d} s+\lambda\,e.\end{displaymath}
Therefore, applying $\gamma$ on both sides,
\begin{align*}
\rho=\gamma(u)&\ge\int_0^1
\psi_\gamma(s)g(s)f(s,u(s))\operatorname{d} s+ \lambda\gamma(e) \ge
\rho f_\rho\int_0^1\psi_\gamma(s)g(s)\operatorname{d} s>\rho,
\end{align*}
which is a contradiction.
\end{proof}

Now we can combine the above Lemmas to prove the following Theorem. The proof of such is straightforward from the properties of the fixed point index stated in Lemma \ref{lemind}.
\begin{thm}
\label{thmmsol1} The integral equation \eqref{eqhamm} has at least one non-zero solution
in $K$ if either of the following conditions hold.

\begin{enumerate}

\item[$(S_{1})$] There exist $\rho _{1},\rho _{2}\in (0,+\infty )$ with $\rho_2
>b(\rho_1)$ such that $(\mathrm{I}_{\rho _{1}}^{0})$ and $(\mathrm{I}_{\rho _{2}}^{1})$ hold.

\item[$(S_{2})$] There exist $\rho _{1},\rho _{2}\in (0,+\infty )$ with $\rho
_{2}>c(\rho _{1})$ such that $(\mathrm{I}_{\rho _{1}}^{1})$ and $(\mathrm{I}%
_{\rho _{2}}^{0})$ hold.
\end{enumerate}
The integral equation \eqref{eqhamm} has at least two non-zero solutions in $K$ if one of
the fo\-llo\-wing conditions hold.

\begin{enumerate}

\item[$(S_{3})$] There exist $\rho _{1},\rho _{2},\rho _{3}\in (0,+\infty )$
with $\rho_2>b(\rho_1)$ and $\rho_3>c(\rho_2)$ such that $(\mathrm{I}_{\rho
_{1}}^{0}),$ $(
\mathrm{I}_{\rho _{2}}^{1})$ $\text{and}\;\;(\mathrm{I}_{\rho _{3}}^{0})$
hold.

\item[$(S_{4})$] There exist $\rho _{1},\rho _{2},\rho _{3}\in (0,+\infty )$
with $\rho_2>c(\rho_1)$ and $\rho_3>b(\rho_2)$ such that $(\mathrm{I}%
_{\rho _{1}}^{1}),\;\;(\mathrm{I}_{\rho _{2}}^{0})$ $\text{and}\;\;(\mathrm{I%
}_{\rho _{3}}^{1})$ hold.
\end{enumerate}

\end{thm}
\begin{rem}
The list of conditions can be extended to obtain more multiplicity results (cf. Lan~\cite{Lan}).
\end{rem}
\section{An example}
We finish this paper with an example to illustrate the applications of Theorem~\ref{thmmsol1}. \\

\emph{Consider the problem
\begin{equation}\label{pronorm} -u''(t)=f(t,u(t)):=\frac{4}{|u(t)|+4},\ t\in I, \quad u(0)=u(1)=0.\end{equation}
Is there any concave solution of problem \eqref{pronorm} satisfying that $\displaystyle{\int_0^1u(s)\operatorname{d} s\ge \frac{1}{20}}$ and $\|u\|_2\le \dfrac{1}{2}$?}\par
To answer this question we will work on a cone $K_\alpha $ of the type given in Example~\ref{conex}, where
\[\alpha (u):=\min\left\{\inf_{t,s\in I}\left[u\left(\frac{t+s}{2}\right)-\frac{u(t)+u(s)}{2}\right],\,u(0),\,-u(0),\,u(1),\,-u(1)\right\}.\]
$K_\alpha $ is precisely the cone of continuous concave functions that vanish at $0$ and $1$. Moreover, this cone is contained in the cone of nonnegative continuous functions.
Observe that we can rewrite problem \eqref{pronorm} in terms of a fixed--point problem for the operator
\begin{equation*}
 u(t)= \int_0^1k(t,s)f(s,u(s))\operatorname{d} s,\end{equation*}
where
\begin{displaymath}k(t,s):=\begin{cases} s(1-t), & 0\le s\le t\le 1,\\
t(1-s), & 0\le t\le s\le 1.\end{cases}\end{displaymath}

 Notice that $k$ is continuous, non-negative, $k(0,s)=k(1,s)=0$ for every $s\in I$ and the function $k(\cdot,s)$ is concave for every $s\in I$ since it is piecewise defined as two line segments, one increasing in the first part of the interval and the other decreasing in the second part. Moreover, $k(\cdot,s)$ is a.e. differentiable with uniformly bounded derivative. Hence, conditions $(C_1)$ and $(C_2)$ are satisfied.\par
 In this case $g\equiv 1$ and $\psi_\alpha \equiv 0$, so $(C_3)$ is also satisfied. Furthermore, $(C_4)$ is satisfied by the definition of $f$ and $(C_5)$ holds since $\psi_\alpha \equiv 0$ and $\alpha$ is non-negative in $K_\alpha$.\par
On the other hand, if we take $\beta(u)=\|u\|_2$ and $\gamma(u)=\displaystyle{\int_0^1u(s)\operatorname{d} s}$ we have that
\[\psi_\beta(s)=\frac{1}{\sqrt{3}} \, s(1-s),\ \psi_\gamma(s)=\frac{1}{2} \, s(1-s) \ \mbox{ for all }s\in I.\]
Hence,
\[\int_0^1\psi_\beta(s)g(s)\operatorname{d} s=\frac{1}{6\sqrt{3}},\ \int_0^1\psi_\gamma(s)g(s)\operatorname{d} s=\frac{1}{12} \ \mbox{ for all } s\in I.\]
Thus, $\beta$ and $\gamma$ satisfy $(C_6)$. Now, $\psi_\gamma(s)\ge 0$ for every $s\in I$, so condition $(C_7)$ is also satisfied. Observe that, since the functions in $K_\alpha $ are nonnegative, $\gamma(u)=\|u\|_1$ for $u\in K_\alpha $.\par
Now, in order to check that $(C_8)$ also holds, we construct the functions $b$ and $c$ using some inequalities comprising the norms  $\|\cdot\|_1$,  $\|\cdot\|_2$ and $\|\cdot\|_\infty$.\par
First, it is a known fact that $\|u\|_1\le\|u\|_2\le\|u\|_\infty$ for functions $u\in\mathcal{C}(I)$, so we can choose $c(\rho)=\rho$. Now, for $u\in K_\alpha $, take $t_u:=\inf\{t\in I\ :\ u(t)=\|u\|_\infty\}$ and define
\[\widetilde u(t):=\begin{dcases} \frac{\|u\|_\infty}{t_u}t, & t\in[0,t_u],\\ \frac{\|u\|_\infty}{1-t_u}(1-t), & t\in[0,t_u].\end{dcases}\]
We have that $u,\widetilde u \in K_\alpha $ and $\widetilde u\le u$. Therefore,
\[\|u\|_1\ge\|\widetilde u\|_1=\frac{1}{2}t_u\|u\|_\infty+\frac{1}{2}(1-t_u)\|u\|_\infty=\frac{1}{2}\|u\|_\infty\ge \frac{1}{2}\|u\|_2.\]
Hence, it is enough to choose $b(\rho)=2\rho$ to guarantee that $(C_8)$ is satisfied.\par

Finally, if we take $\rho_1=\dfrac{1}{20}$ and $\rho_2=\dfrac{1}{2}$ then we have that $\rho_2>b(\rho_1)$. Observe that
\[f_{\rho_1}\ge\frac{f(t,u(t))}{\rho_1}=\frac{80}{|u(t)|+4}\ge\frac{80}{\|u\|_\infty+4}\ge \frac{80}{2\|u\|_1+4}= \frac{80}{2/20+4}=\frac{800}{41},\]
for $t\in I$ and $\gamma(u)=\|u\|_1=\rho_1$.
Hence,
\[f_{\rho_1}\int_0^1\psi_\gamma(s)g(s)\operatorname{d} s\ge\frac{800}{41}\frac{1}{12}=\frac{200}{123}>1.\]
Therefore, condition $(I_{\rho_1}^0)$ holds.\par On the other hand,
\[f^{\rho_2}\le\frac{f(t,u(t))}{\rho_2}=\frac{8}{|u(t)|+4}\ge\frac{8}{4}=2,\]
for $t\in I$ and $\gamma(u)=\|u\|_2=\rho_2$ Thus,
\[f^{\rho_2}\int_0^1\psi_\beta(s)g(s)\operatorname{d} s\ge2\frac{1}{6\sqrt{3}}=\frac{1}{3\sqrt{3}}<1.\]
Therefore, condition $(I_{\rho_2}^1)$ is satisfied. This means that condition $(S_1)$ in Theorem~\ref{thmmsol1} holds and, hence, there exists a solution $u$ of problem \eqref{pronorm} in $K_\alpha ^{\beta,\,\rho_2}\backslash K_\alpha ^{\gamma,\,\rho_1}$. That is, such a solution is concave, nonnegative and satisfies $\displaystyle{\int_0^1u(s)\operatorname{d} s\ge\frac{1}{20}}$ and $\|u\|_2\le \dfrac{1}{2}$.


\newpage

\begin{thebibliography}{10}
\expandafter\ifx\csname urlstyle\endcsname\relax
  \providecommand{\doi}[1]{doi:\discretionary{}{}{}#1}\else
  \providecommand{\doi}{doi:\discretionary{}{}{}\begingroup
  \urlstyle{rm}\Url}\fi
\newcommand{\Capitalize}[1]{\uppercase{#1}}
\newcommand{\capitalize}[1]{\expandafter\Capitalize#1}

\bibitem{Alg}
M.~A. Alghamdi, D.~O'Regan, N.~Shahzad, \emph{Krasnosel'ski\u{\i} type fixed
  point theorems for mappings on nonconvex sets}, Abstr. Appl. Anal.,
  2012(2012).

\bibitem{amann}
H.~Amann, \emph{Fixed point equations and nonlinear eigenvalue problems in
  ordered banach spaces}, SIAM rev., 18(1976), ~4, 620--709.

\bibitem{Anu}
V.~Anuradha, D.~Hai, R.~Shivaji, \emph{Existence results for superlinear
  semipositone bvp's}, Pro\-cee\-dings of the American Mathematical Society,
  124(1996), ~3, 757--763.

\bibitem{AP2001}
R.~Avery, A.~Peterson, \emph{Three positive fixed points of nonlinear operators
  on ordered banach spaces}, Computers \& Mathematics with Applications,
  42(2001), ~3, 313--322.

\bibitem{Berzig2014}
M.~Berzig, S.~Chandok, M.~S. Khan, \emph{Generalized krasnosel'ski\u{\i} fixed
  point theorem involving auxiliary functions in bimetric spaces and
  application to two-point boundary value problem}, Appl. Math. Comput.,
  248(2014), 323--327.

\bibitem{CAC}
A.~Cabada, J.~{\'A}. Cid, G.~Infante, \emph{New criteria for the existence of
  non-trivial fixed points in cones}, Fixed Point Theory and Applications,
  2013(2013), ~1, 1--12.

\bibitem{CCI}
---{}---{}---, \emph{A positive fixed point theorem with applications to
  systems of hammerstein integral equations}, Boundary Value Problems,
  2014(2014), ~1, 1--10.

\bibitem{Cai}
G.~Cai, S.~Bu, \emph{Krasnoselskii-type fixed point theorems with applications
  to hammerstein integral equations in l1 spaces}, Mathematische Nachrichten,
  286(2013),  14-15, 1452--1465.

\bibitem{Chi}
C.~Chidume, Y.~Shehu, \emph{Approximation of solutions of equations of
  hammerstein type in hilbert spaces}, Fixed Point Theory, 16(2015),
  ~1, 91--101.

\bibitem{CITZ}
J.~Cid, G.~Infante, M.~Tvrd{\`y}, M.~Zima, \emph{A topological approach to
  periodic oscillations related to the liebau phenomenon}, Journal of
  Mathematical Analysis and Applications, 423(2015), ~2, 1546--1556.

\bibitem{Dha}
B.~Dhage, S.~Ntouyas, \emph{A krasnosel'ski\u{\i} nonlinear alternative type
  fixed point theorem with applications to nonlinear integral equations},
  Indian J. Maths., 56(2014), ~1, 113--124.

\bibitem{Don}
Y.~Dongming, Z.~Qiang, P.~Zhigang, \emph{Existence of positive solutions for
  neumann boundary value problem with a variable coefficient}, Int. J. Differ.
  Equ., 2011(2011).

\bibitem{Donoghue}
W.~F. Donoghue, \emph{Distributions and Fourier transforms}, Academic Press
  (1969).

\bibitem{Erbe}
L.~Erbe, \emph{Eigenvalue criteria for existence of positive solutions to
  nonlinear boundary value problems}, Mathematical and computer modelling,
  32(2000), ~5, 529--539.

\bibitem{Franco}
D.~Franco, G.~Infante, J.~Per{\'a}n, \emph{A new criterion for the existence of
  multiple solutions in cones}, Proceedings of the Royal Society of Edinburgh, 
  Section A Mathematics, 142(2012), ~05, 1043--1050.

\bibitem{Goodrich3}
C.~S. Goodrich, \emph{On nonlocal bvps with nonlinear boundary conditions with
  asymptotically sublinear or superlinear growth}, Math. Nachr., 285(2012),
   11-12, 1404--1421.

\bibitem{Goodrich4}
---{}---{}---, \emph{Positive solutions to boundary value problems with
  nonlinear boundary conditions}, Nonlinear Anal., 75(2012),
  ~1, 417--432.

\bibitem{Goodrich8}
---{}---{}---, \emph{On a nonlocal bvp with nonlinear boundary conditions},
  Results in Mathematics, 63(2013),  3-4, 1351--1364.

\bibitem{Goodrich5}
---{}---{}---, \emph{On nonlinear boundary conditions satisfying certain
  asymptotic behavior}, Nonlinear Anal., 76(2013), 58--67.

\bibitem{guolak}
D.~Guo, V.~Lakshmikantham, \emph{Nonlinear problems in abstract cones},
  Academic press (2014).

\bibitem{gi-pp1}
G.~Infante, P.~Pietramala, \emph{Nonlocal impulsive boundary value problems
  with solutions that change sign}, \emph{Mathematical Models in
  Engineering, Biology and Medicine. Conference on Boundary Value Problems.
  September 16-19, 2008, Santiago de Compostela, Spain.}, 22.

\bibitem{gi-pp}
---{}---{}---, \emph{Perturbed hammerstein integral inclusions with solutions
  that change sign}, Comment. Math. Univ. Carolin., 50(2009),
  ~4, 591--605.

\bibitem{IP-disp}
---{}---{}---, \emph{The displacement of a sliding bar subject to nonlinear
  controllers},  \emph{Differential and Difference Equations with
  Applications}, Springer (2013), 429--437.

\bibitem{gippmt}
G.~Infante, P.~Pietramala, M.~Tenuta, \emph{Existence and localization of
  positive solutions for a nonlocal bvp arising in chemical reactor theory},
  Commun. Nonlinear Sci. Numer. Simul., 19(2014), ~7, 2245--2251.

\bibitem{gi-pp-ft}
G.~Infante, P.~Pietramala, F.~A.~F. Tojo, \emph{Nontrivial solutions of local
  and nonlocal neumann boundary value problems}, Proc. Edinb. Math. Sect. A.,
  (To appear).

\bibitem{IPZ}
G.~Infante, P.~Pietramala, M.~Zima, \emph{Positive solutions for a class of
  nonlocal impulsive bvps via fixed point index}, Topological Methods in
  Nonlinear Analysis, 36(2010), ~2, 263--284.

\bibitem{gijwjiea}
G.~Infante, J.~Webb, \emph{Three point boundary value problems with solutions
  that change sign}, J. Integral Equations Appl., 15(2003), 37--57.

\bibitem{gijwems}
---{}---{}---, \emph{Nonlinear non-local boundary-value problems and perturbed
  hammerstein integral equations}, Proc. Edinb. Math. Soc. (2), 49(2006),
  ~03, 637--656.

\bibitem{Jan4}
T.~Jankowski, \emph{Nonnegative solutions to nonlocal boundary value problems
  for systems of second-order differential equations dependent on the
  first-order derivatives}, Nonlinear Anal., 87(2013), 83--101.

\bibitem{JLY2012}
J.~Jiang, L.~Liu, Y.~Wu, \emph{Positive solutions for second-order singular
  semipositone differential equations involving stieltjes integral conditions},
   \emph{Abstract and Applied Analysis},  2012, Hindawi
  Publishing Corporation.

\bibitem{Kras60}
M.~Krasnosel'ski\u{\i}, \emph{Fixed points of cone-compressing or
  cone-extending operators}, Soviet Mathematics. Doklady, 1(1960), 1285--1288.

\bibitem{Kwo}
M.~K. Kwong, \emph{On krasnosel'ski\u{\i}'s cone fixed point theorem}, Fixed
  Point Theory and Appl., 2008(2008).

\bibitem{Lan}
K.~Lan, \emph{Multiple positive solutions of hammerstein integral equations
  with singularities}, Differ. Equ. Dyn. Syst., 8(2000), 175--195.

\bibitem{Lan2}
---{}---{}---, \emph{Eigenvalues of semi-positone hammerstein integral
  equations and applications to boun\-da\-ry value problems}, Nonlinear Analysis, 
  Theory, Methods \& Applications, 71(2009), ~12, 5979--5993.

\bibitem{Li-Cong}
Q.~Li, F.~Cong, D.~Jiang, \emph{Multiplicity of positive solutions to second
  order neumann boundary value problems with impulse actions}, Appl. Math.
  Comput., 206(2008), ~2, 810--817.

\bibitem{ma-non}
R.~Ma, \emph{Nonlinear periodic boundary value problems with sign-changing
  green's function}, Nonlinear Analysis,  Theory, Methods \& Applications,
  74(2011), ~5, 1714--1720.

\bibitem{Pet}
I.-R. Petre, A.~Petrusel, \emph{Krasnosel'ski\u{\i}'s theorem in generalized
  banach spaces and application}, Electro. J. Qual. Theory Differ. Equ.,
  2012(2012), ~85, 1--20.

\bibitem{sun2}
J.-P. Sun, W.-T. Li, \emph{Multiple positive solutions to second-order neumann
  boundary value pro\-blems}, Appl. Math. Comput., 146(2003),
  ~1, 187--194.

\bibitem{Sun}
J.-P. Sun, W.-T. Li, S.~S. Cheng, \emph{Three positive solutions for
  second-order neumann boundary value problems}, Applied Mathematics Letters,
  17(2004), ~9, 1079--1084.

\bibitem{Sun3}
Y.~Sun, Y.~J. Cho, D.~O'Regan, \emph{Positive solutions for singular second
  order neumann boundary value problems via a cone fixed point theorem}, Appl.
  Math. Comput., 210(2009), ~1, 80--86.

\bibitem{Wan}
F.~Wang, Y.~Cui, F.~Zhang, \emph{A singular nonlinear second-order neumann
  boundary value problem with positive solutions}, Thai J. Math., 7(2012),
  ~2, 243--257.

\bibitem{WZFP}
F.~Wang, F.~Zhang, \emph{An extension of fixed point theorems concerning cone
  expansion and compression and its application}, Communications of the Korean
  Mathematical Society, 24(2009), ~2, 281--290.

\bibitem{Wang1}
---{}---{}---, \emph{Existence of positive solutions of neumann boundary value
  problem via a cone compression-expansion fixed point theorem of functional
  type}, J. Appl. Math. Comput., 35(2011),  1-2, 341--349.

\bibitem{WebbUni}
J.~Webb, \emph{Uniqueness of the principal eigenvalue in nonlocal boundary
  value problems}, Discrete Contin. Dyn. Syst. Ser. S, 1(2008),
  ~1, 177--186.

\bibitem{jw-tmna}
---{}---{}---, \emph{A class of positive linear operators and applications to
  nonlinear boundary value pro\-blems}, Topol. Methods Nonlinear Anal., 39(2012),
  ~2, 221--242.

\bibitem{jw-gi-jlms}
J.~Webb, G.~Infante, \emph{Positive solutions of nonlocal boundary value
  problems,  a unified approach}, J. Lond. Math. Soc., 74(2006),
  ~3, 673--693.

\bibitem{jw-gi-jlmsII}
---{}---{}---, \emph{Non-local boundary value problems of arbitrary order}, J.
  Lond. Math. Soc., 79(2009), ~1, 238--258.

\bibitem{webb-semi}
---{}---{}---, \emph{Semi-positone nonlocal boundary value problems of
  arbitrary order}, Communications on Pure and Applied Analysis, 9(2009),
  ~2, 563--581.

\bibitem{WIF}
J.~Webb, G.~Infante, D.~Franco, \emph{Positive solutions of nonlinear
  fourth-order boundary-value problems with local and non-local boundary
  conditions}, Proceedings of the Royal Society of Edinburgh,  Section A
  Mathematics, 138(2008), ~02, 427--446.

\bibitem{jwkleig}
J.~Webb, K.~Lan, \emph{Eigenvalue criteria for existence of multiple positive
  solutions of nonlinear boundary value problems of local and nonlocal type},
  Topol. Methods Nonlinear Anal., 27(2006), 91--115.

\bibitem{jwmz-na}
J.~Webb, M.~Zima, \emph{Multiple positive solutions of resonant and
  non-resonant nonlocal boundary value problems}, Nonlinear Anal., 71(2009),
  ~3, 1369--1378.

\bibitem{ZA2011}
S.~Zhong, Y.~An, \emph{Existence of positive solutions to periodic boundary
  value problems with sign-changing green's function}, Boundary Value Problems,
  2011(2011), ~1, 1--6.

\end{thebibliography}
\end{document}